\documentclass{article}
\usepackage{amsfonts}
\usepackage{amsmath}

\newtheorem{theorem}{Theorem}

\newtheorem{corollary}[theorem]{Corollary}

\newtheorem{definition}[theorem]{Definition}

\newtheorem{lemma}[theorem]{Lemma}

\newtheorem{proposition}[theorem]{Proposition}

\newenvironment{proof}[1][Proof]{\noindent\textbf{#1.} }{\ \rule{0.5em}{0.5em}}
\input{tcilatex}
\begin{document}

\title{\textbf{On the some properties of circulant matrix with third order
linear recurrent sequence}}
\author{Arzu Coskun and Necati Taskara\thanks{%
e mail: \ \textit{arzucoskun58@gmail.com, ntaskara@selcuk.edu.tr}} \\
%EndAName
Department of Mathematics, Science Faculty, \\
Selcuk University, Campus, 42075, Konya, Turkey}
\maketitle

\begin{abstract}
In this paper, firstly, we give the some fundamental properties of Van Der
Laan numbers. After, we define the circulant matrices $C(Z)$ which entries
is third order linear recurrent sequence. In addition, we compute
eigenvalues, spectral norm and determinant of this matrix. Consequently, by
using properties of this sequence, we obtain the eigenvalues, norms and
determinants of circulant matrices with Cordonnier, Perrin and Van Der Laan
numbers.

\textit{Keywords}: Third order linear recurrent sequence, Cordonnier
numbers, Perrin numbers, Van Der Laan numbers, circulant matrix, norm,
determinant.
\end{abstract}

\section{Introduction}

Shannon, Horadam and Anderson \cite{4} defined respectively Cordonnier
sequence $\left\{ P_{n}\right\} _{n\in 
%TCIMACRO{\U{2115} }%
%BeginExpansion
\mathbb{N}
%EndExpansion
}$\ as 
\begin{equation}
P_{n+3}=P_{n+1}+P_{n},\ \ \ \ P_{0}=1,~P_{1}=1,\ P_{2}=1,  \label{1}
\end{equation}%
Perrin sequence $\left\{ Q_{n}\right\} _{n\in 
%TCIMACRO{\U{2115} }%
%BeginExpansion
\mathbb{N}
%EndExpansion
}$\ as

\begin{equation}
Q_{n+3}=Q_{n+1}+Q_{n},\ \ \ \ Q_{0}=3,~Q_{1}=0,\ Q_{2}=2,  \label{2}
\end{equation}%
and Van Der Laan sequence $\left\{ R_{n}\right\} _{n\in 
%TCIMACRO{\U{2115} }%
%BeginExpansion
\mathbb{N}
%EndExpansion
}$\ as

\begin{equation}
R_{n+3}=R_{n+1}+R_{n},\ \ \ \ R_{0}=0,~R_{1}=1,\ R_{2}=0.  \label{3}
\end{equation}

The fact that the Cordonnier numbers and Perrin numbers are a linear
combination of $\alpha ^{n}$, $\beta ^{n}$ and $\gamma ^{n}$, that is,%
\begin{equation}
P_{n}=\frac{\alpha ^{n+4}}{\left( \alpha -\beta \right) \left( \alpha
-\gamma \right) }+\frac{\beta ^{n+4}}{\left( \beta -\alpha \right) \left(
\beta -\gamma \right) }+\frac{\gamma ^{n+4}}{\left( \gamma -\alpha \right)
\left( \gamma -\beta \right) },  \label{1.2}
\end{equation}

\begin{equation}
Q_{n}=\alpha ^{n}+\beta ^{n}+\gamma ^{n}.  \label{2.2}
\end{equation}

Hence, the relations are hold%
\begin{equation}
\alpha +\beta +\gamma =0,\;\alpha \beta \gamma =1,\ \alpha \beta +\beta
\gamma +\alpha \gamma =-1,  \label{1.1}
\end{equation}%
where $\alpha $, $\beta $ and $\gamma $ are roots of equations (\ref{1}), (%
\ref{2}), (\ref{3}).\qquad

Elia \cite{12}, gave third order linear recurrent sequence$\left\{
T_{0},T_{1},T_{2},...\right\} $ defined by the recurrence

\begin{equation}
T_{n+3}=pT_{n+2}+qT_{n+1}+rT_{n},\ \ \ \ T_{0}=a,~T_{1}=b,\ T_{2}=c.
\label{5}
\end{equation}%
Also, he studied Tribonacci cubic form and solve the integer representation
problem for the Tribonacci cubic form.

Recently, it has been widely studied the some properties of the circulant
matrices with special numbers. For instance, in \cite{6}, they defined the
generalized $k$-Horadam numbers and computed the spectral norm, eigenvalues
and the determinant of circulant matrix with this numbers. Solak \cite{7}
defined \ the $n\times n$ circulant matrices $A=\left[ a_{ij}\right] $ and $%
B=\left[ b_{ij}\right] $, where $a_{ij}\equiv F_{\left( \func{mod}\left(
j-i,n\right) \right) }$ and $b_{ij}\equiv L_{\left( \func{mod}\left(
j-i,n\right) \right) }$. Additionally, he investigated the upper and lower
bounds of the matrices $A$ and $B$, respectively. In \cite{5}, Ipek obtained
the spectral norms of circulant matrices $A=\left[ a_{ij}\right] $ and $B=%
\left[ b_{ij}\right] $, where $a_{ij}\equiv F_{\left( \func{mod}\left(
j-i,n\right) \right) }$ and $b_{ij}\equiv L_{\left( \func{mod}\left(
j-i,n\right) \right) }$. Shen and Cen, in \cite{8,11}, have found upper and
lower bounds for the spectral norms of $r$-circulant matrices and obtained
some bounds for the spectral norms of Kronecker and Hadamard products of
these matrices. Also, they gave the determinants and inverses of circulant
matrices with Fibonacci and Lucas numbers. In \cite{3}, it has been studied
the norms, eigenvalues and determinants of some matrices related to
different numbers. Yazlik \cite{9} obtained upper and lower bounds for the
spectral norm of an $r$-circulant matrices $H=C_{r}\left(
H_{k,0},H_{k,1},H_{k,2},...,H_{k,n-1}\right) $ whose entries are the
generalized $k$-Horadam numbers. Additionally, he find new formulas to
calculate the eigenvalues and determinant of the matrix $H$. In \cite{10},
they defined the circulant matrices with Jacobsthal and Jacobsthal-Lucas
numbers and computed the determinants and inverses of these matrices.

In the light of the above studies, in here, we present some properties of
the Van Der Laan sequence as Binet formula, sum. After, we find eigenvalues,
spectral norm and determinant of circulant matrix with the third order
sequence. Finally, we give eigenvalues, norms and determinants of circulant
matrices with Cordonnier, Perrin and Van Der Laan numbers via properties of
circulant matrix with this third order sequence.

Now, we give some preliminaries about circulant matrix and the spectral norm
of a matrix.

The circulant matrix $C=\left[ c_{ij}\right] \in M_{n,n}\left( 
%TCIMACRO{\U{2102} }%
%BeginExpansion
\mathbb{C}
%EndExpansion
\right) $ is defined by the form 
\begin{equation*}
c_{ij}=\{%
\begin{array}{c}
c_{j-i},\ \ \ \ \ \ \ j\geq i \\ 
c_{n+j-i},\ \ \ j<i%
\end{array}%
\end{equation*}

The spectral norm of $A$, for a matrix $A=\left[ a_{i,j}\right] \in M_{m,n}(%
%TCIMACRO{\U{2102} }%
%BeginExpansion
\mathbb{C}
%EndExpansion
)$, is given by%
\begin{equation*}
\left\Vert A\right\Vert _{2}=\sqrt{\max\limits_{1\leq i\leq n}\lambda
_{i}\left( A^{\ast }A\right) },
\end{equation*}%
where $A^{\ast }$ is the conjugate transpose of matrix $A$.

\begin{lemma}
\bigskip \lbrack 2] Let $A=circ(a_{0},a_{1},\cdots ,a_{n-1})$ be a $n\times
n $ circulant matrix. Then we have%
\begin{equation*}
\lambda _{j}\left( A\right) =\sum_{k=0}^{n-1}a_{k}w^{-jk},
\end{equation*}%
where $w=e^{\frac{2\pi i}{n}},\ i=\sqrt{-1},\ j=0,1,\ldots n-1.$
\end{lemma}

\begin{lemma}
\lbrack 1] Let $A$ be an $n\times n$ matrix with eigenvalues $\lambda
_{1},\lambda _{2},\ldots ,\lambda _{n}.$ Then, $A$ is a normal matrix if and
only if the eigenvalues of $AA^{\ast }\ $are $\left\vert \lambda
_{1}\right\vert ^{2},\left\vert \lambda _{2}\right\vert ^{2},\ldots
,\left\vert \lambda _{n}\right\vert ^{2},$ where $A^{\ast }$ is the
conjugate transpose of the matrix $A$.
\end{lemma}

\section{Circulant Matrices with Third Order Linear Sequences}

Firstly, since some of results of this paper concern about the spectral
norm, eigenvalues and determinant of the circulant matrix entried by the Van
Der Laan numbers, we need to introduce some properties of this sequence.

The Binet formula of Van Der Laan sequence, for every $n\in 
%TCIMACRO{\U{2115} }%
%BeginExpansion
\mathbb{N}
%EndExpansion
$, is obtained as%
\begin{equation}
R_{n}=\frac{\alpha ^{n}}{\left( \alpha -\beta \right) \left( \alpha -\gamma
\right) }+\frac{\beta ^{n}}{\left( \beta -\alpha \right) \left( \beta
-\gamma \right) }+\frac{\gamma ^{n}}{\left( \gamma -\alpha \right) \left(
\gamma -\beta \right) },  \label{3.2}
\end{equation}%
\ \ \ \ \qquad\ where $\alpha $, $\beta $ and $\gamma $ are roots of
equation (\ref{3}).

Also, for $n\geq 1$, we have%
\begin{equation*}
\overset{n}{\underset{k=0}{\dsum }}R_{k}=R_{n+5}-1,
\end{equation*}%
\begin{equation*}
\overset{n}{\underset{k=0}{\dsum }}%
R_{k}^{2}=R_{n+2}^{2}-R_{n-1}^{2}-R_{n-3}^{2}+1.
\end{equation*}

If we take $p=0,q=1,r=1$ in (\ref{5}), we obtain as

\begin{equation}
Z_{n+3}=Z_{n+1}+Z_{n},  \label{4}
\end{equation}%
where initial conditions $Z_{0}=a,~Z_{1}=b,\ Z_{2}=c$ . Also, the Binet
formula of $Z_{n}$ sequence, for every $n\in 
%TCIMACRO{\U{2115} }%
%BeginExpansion
\mathbb{N}
%EndExpansion
$, is obtained as

\begin{equation}
Z_{n}=\frac{\left( \alpha ^{2}-1\right) a+\alpha b+c}{\left( \alpha -\beta
\right) \left( \alpha -\gamma \right) }\alpha ^{n}+\frac{\left( \beta
^{2}-1\right) a+\beta b+c}{\left( \beta -\alpha \right) \left( \beta -\gamma
\right) }\beta ^{n}+\frac{\left( \gamma ^{2}-1\right) a+\gamma b+c}{\left(
\gamma -\alpha \right) \left( \gamma -\beta \right) }\gamma ^{n},
\label{4.2}
\end{equation}%
where $\alpha $, $\beta $ and $\gamma $ are roots of equation $x^{3}-x-1=0.$

\begin{proposition}
For $n\geq 1$, the following relations are hold:%
\begin{equation*}
\overset{n}{\underset{k=0}{\dsum }}Z_{k}=Z_{n+5}-Z_{4},
\end{equation*}%
\begin{equation*}
\overset{n}{\underset{k=0}{\dsum }}%
Z_{k}^{2}=Z_{n+2}^{2}-Z_{n-1}^{2}-Z_{n-3}^{2}+T
\end{equation*}%
where $T=2a\left( a-c\right) -\left( b+c\right) ^{2}$.
\end{proposition}

We formulate eigenvalues, spectral norms and determinants of the circulant
matrices with the third order linear sequences. In order to do that, we can
define \textit{the circulant matrix} as follows.

\begin{definition}
$n\times n$ circulant matrix with third order linear sequence entries is
defined by 
\begin{equation}
C(Z)=\left[ 
\begin{array}{ccccc}
Z_{0} & Z_{1} & Z_{2} & \cdots & Z_{n-1} \\ 
Z_{n-1} & Z_{0} & Z_{1} & \cdots & Z_{n-2} \\ 
Z_{n-2} & Z_{n-1} & Z_{0} & \cdots & Z_{n-3} \\ 
\vdots & \vdots & \vdots & \ddots & \vdots \\ 
Z_{1} & Z_{2} & Z_{3} & \cdots & Z_{0}%
\end{array}%
\right] .  \label{1.3}
\end{equation}
\end{definition}

The following theorem gives us the eigenvalues of circulant matrix with
third order linear recurrent sequence.

\begin{theorem}
Let $C(Z)=circ(Z_{0},Z_{1},...,Z_{n-1})$ be circulant matrix. Then the
eigenvalues of $C(Z)\ $are%
\begin{equation*}
\lambda _{j}\left( C(Z)\right) =\frac{Z_{n}-a+\left( Z_{n+1}-b\right)
w^{-j}+\left( Z_{n-1}-c+a\right) w^{-2j}}{w^{-3j}+w^{-2j}-1},
\end{equation*}%
where $w=e^{\frac{2\pi i}{n}},\ i=\sqrt{-1},\ j=0,1,\ldots ,n-1$.
\end{theorem}

\begin{proof}
By taking $X=\left( \alpha -\beta \right) \left( \alpha -\gamma \right) $,$\
Y=\left( \beta -\alpha \right) \left( \beta -\gamma \right) $,$\ W=\left(
\gamma -\alpha \right) \left( \gamma -\beta \right) ,\ $from Lemma 1, we have%
\begin{eqnarray*}
\lambda _{j}\left( C(Z)\right) &=&\sum_{k=0}^{n-1}Z_{k}w^{-jk} \\
&=&\sum_{k=0}^{n-1}\left( 
\begin{array}{c}
\frac{\left( \alpha ^{2}-1\right) a+\alpha b+c}{X}\alpha ^{k}+\frac{\left(
\beta ^{2}-1\right) a+\beta b+c}{Y}\beta ^{k} \\ 
+\frac{\left( \gamma ^{2}-1\right) a+\gamma b+c}{W}\gamma ^{k}%
\end{array}%
\right) w^{-jk} \\
&=&%
\begin{array}{c}
\frac{\left( \alpha ^{2}-1\right) a+\alpha b+c}{X}\left( \frac{\left( \alpha
w^{-j}\right) ^{n}-1}{\alpha w^{-j}-1}\right) +\frac{\left( \beta
^{2}-1\right) a+\beta b+c}{Y}\left( \frac{\left( \beta w^{-j}\right) ^{n}-1}{%
\beta w^{-j}-1}\right) \\ 
+\frac{\left( \gamma ^{2}-1\right) a+\gamma b+c}{W}\left( \frac{\left(
\gamma w^{-j}\right) ^{n}-1}{\gamma w^{-j}-1}\right)%
\end{array}
\\
&=&%
\begin{array}{c}
\frac{\left( \alpha ^{n}-1\right) \left[ \left( \alpha ^{2}-1\right)
a+\alpha b+c\right] }{X\left( \alpha w^{-j}-1\right) } \\ 
+\frac{\left( \beta ^{n}-1\right) \left[ \left( \beta ^{2}-1\right) a+\beta
b+c\right] }{Y\left( \beta w^{-j}-1\right) }+\frac{\left( \gamma
^{n}-1\right) \left[ \left( \gamma ^{2}-1\right) a+\gamma b+c\right] }{%
W\left( \gamma w^{-j}-1\right) }%
\end{array}%
.
\end{eqnarray*}%
By considering (\ref{1.1}) and (\ref{4.2}), we obtain%
\begin{equation*}
\lambda _{j}\left( C(Z)\right) =\frac{Z_{n}-a+\left( Z_{n+1}-b\right)
w^{-j}+\left( Z_{n-1}-c+a\right) w^{-2j}}{w^{-3j}+w^{-2j}-1}
\end{equation*}%
which is desired.
\end{proof}

By using Theorem 5, we obtain eigenvalues of circulant matrices with
Cordonnier, Perrin and Van Der Laan sequences as follows.

\begin{corollary}
$i)$ Let $C(P)=circ(P_{0},P_{1},...,P_{n-1})$ be circulant matrix. If we
take $a=b=c=1$ in Theorem 5, then the eigenvalues of $C(P)\ $are%
\begin{equation*}
\lambda _{j}\left( C(P)\right) =\frac{P_{n}-1+\left( P_{n+1}-1\right)
w^{-j}+P_{n-1}w^{-2j}}{w^{-3j}+w^{-2j}-1},
\end{equation*}%
where $P_{n}$ is the $n$th Cordonnier number and $w=e^{\frac{2\pi i}{n}},\ i=%
\sqrt{-1},\ j=0,1,\ldots ,n-1$.
\end{corollary}

$ii)$Let $C(Q)$ be circulant matrix with Perrin numbers. If we take $%
a=3,b=0,c=2$, the eigenvalues of $C(Q)\ $are%
\begin{equation*}
\lambda _{j}\left( C(Q)\right) =\frac{Q_{n}-3+Q_{n+1}w^{-j}+\left(
Q_{n-1}+1\right) w^{-2j}}{w^{-3j}+w^{-2j}-1},
\end{equation*}%
where $Q_{n}$ is the $n$th Perrin number and $w=e^{\frac{2\pi i}{n}},\ i=%
\sqrt{-1},\ j=0,1,\ldots ,n-1$.

$iii)$Let $C(R)=circ(R_{0},R_{1},...,R_{n-1})$ be circulant matrix. If we
take $a=c=0,b=1$ in theorem 5 the eigenvalues of $C(R)\ $are%
\begin{equation*}
\lambda _{j}\left( C(R)\right) =\frac{R_{n}+\left( R_{n+1}-1\right)
w^{-j}+R_{n-1}w^{-2j}}{w^{-3j}+w^{-2j}-1},
\end{equation*}%
where $R_{n}$ is the $n$th Van Der Laan number and $w=e^{\frac{2\pi i}{n}},\
i=\sqrt{-1},\ j=0,1,\ldots ,n-1.$

Because matrix $C(Z)$ is normal matrix, we can write $\lambda _{j}\left(
C(Z)C(Z)^{\ast }\right) =\left\vert \lambda _{j}\left( C(Z)\right)
\right\vert ^{2}$. Also, since the matrices $C(P)$, $C(Q)$ and $C(R)$ are
normal matrices, we can write $\lambda _{j}\left( C(P)C(P)^{\ast }\right)
=\left\vert \lambda _{j}\left( C(P)\right) \right\vert ^{2}$, $\lambda
_{j}\left( C(Q)C(Q)^{\ast }\right) =\left\vert \lambda _{j}\left(
C(Q)\right) \right\vert ^{2}$ and $\lambda _{j}\left( C(R)C(R)^{\ast
}\right) =\left\vert \lambda _{j}\left( C(R)\right) \right\vert ^{2}$. Then,
we have the following theorem and corollary that deal with spectral norm.

\begin{theorem}
Let $C(Z)$ be an $n\times n$ circulant matrix with the $Z_{n}$ entries.
Then, we have
\end{theorem}

\begin{equation*}
\left\Vert C(Z)\right\Vert _{2}=Z_{n+4}-Z_{4}.
\end{equation*}

\begin{proof}
From Lemma 2, we can write%
\begin{equation*}
\left\Vert C(Z)\right\Vert _{2}=\sqrt{\left( \max\limits_{0\leq j\leq
n-1}\left\vert \lambda _{j}\left( C(Z)\right) \right\vert ^{2}\right) }.
\end{equation*}%
In this last equality, for $j=0$, $\lambda _{0}$ \ becomes \ the maximum
eigenvalue. Thus, $\left\Vert C(Z)\right\Vert _{2}=\left\vert \lambda
_{0}\left( C(Z)\right) \right\vert $. Also, from the Theorem 5, we clearly
obtain%
\begin{equation*}
\left\Vert C(Z)\right\Vert _{2}=Z_{n+4}-Z_{4}.
\end{equation*}%
Hence proof is completed.
\end{proof}

For special cases of $a,b,c$ in Theorem 7, we obtain following corollary for
spectral norms of $C(P),C(Q)$ and $C(R)$.

\begin{corollary}
$i)$ Let $C(P)$ be an $n\times n$ circulant matrix with the Cordonnier
numbers entries.  If we take $a=b=c=1$ in (\ref{4}),we obtain Cordonnier
sequence. And so, from Theorem 7, we obtain  
\begin{equation*}
\left\Vert C(P)\right\Vert _{2}=P_{n+4}-2
\end{equation*}%
where $P_{n}$ is $n$th Cordonnier number.
\end{corollary}

$ii)$ Similarly to $i)$, if we take $a=3,b=0,c=2$ in (\ref{4}), we obtain
for $C(Q)$ matrix%
\begin{equation*}
\left\Vert C(Q)\right\Vert _{2}=Q_{n+4}-2.
\end{equation*}%
where $Q_{n}$ is $n$th Perrin number.

$iii)$ Similarly,  if we take $a=c=0,b=1$ in (\ref{4}), we have%
\begin{equation*}
\left\Vert C(R)\right\Vert _{2}=R_{n+4}-1.
\end{equation*}%
where $R_{n}$ is $n$th Van Der Laan number.

Let $A=\left[ a_{ij}\right] $ be an $n\times n$ matrix. Since row norm and
column norm of the $A$ matrix are $\left\Vert A\right\Vert _{\infty }=%
\underset{1\leq i\leq n}{\max }\left( \underset{j=1}{\overset{n}{\dsum }}%
\left\vert a_{ij}\right\vert \right) $ and $\left\Vert A\right\Vert _{1}=%
\underset{1\leq j\leq n}{\max }\left( \underset{i=1}{\overset{n}{\dsum }}%
\left\vert a_{ij}\right\vert \right) $, it can be easily seen that $%
\left\Vert C(Z)\right\Vert _{2}=\left\Vert C(Z)\right\Vert _{1}=\left\Vert
C(Z)\right\Vert _{\infty }.$ Also, $\left\Vert C(P)\right\Vert
_{2}=\left\Vert C(P)\right\Vert _{1}=\left\Vert C(P)\right\Vert _{\infty },$ 
$\left\Vert C(Q)\right\Vert _{2}=\left\Vert C(Q)\right\Vert _{1}=\left\Vert
C(Q)\right\Vert _{\infty }$ and $\left\Vert C(R)\right\Vert _{2}=\left\Vert
C(R)\right\Vert _{1}=\left\Vert C(R)\right\Vert _{\infty }.$

The following theorem gives us the determinant of $C(Z)$.

\begin{theorem}
The determinant \ of \ the matrix $C(Z)=circ\left( Z_{1},Z_{2},\ldots
,Z_{n}\right) $ is written by 
\begin{equation*}
\det (C(Z))=\frac{\left( Z_{n}-a\right) ^{n}\left(
1-K^{n}-L^{n}+K^{n}L^{n}\right) }{\left( -1\right) ^{n}\left(
Q_{-n}-Q_{n}\right) },
\end{equation*}%
where $K=\dfrac{b-Z_{n+1}-\sqrt{\left( Z_{n+1}-b\right) ^{2}-4\left(
Z_{n}-a\right) \left( Z_{n-1}-c+a\right) }}{2\left( Z_{n}-a\right) }$,
\end{theorem}

$L=\dfrac{b-Z_{n+1}+\sqrt{\left( Z_{n+1}-b\right) ^{2}-4\left(
Z_{n}-a\right) \left( Z_{n-1}-c+a\right) }}{2\left( Z_{n}-a\right) }$.

\begin{proof}
From Theorem 5, we have 
\begin{eqnarray*}
\det (C(Z)) &=&\prod\limits_{j=0}^{n-1}\lambda _{j}\left( C(Z)\right) \\
&=&\prod\limits_{j=0}^{n-1}\frac{Z_{n}-a+\left( Z_{n+1}-b\right)
w^{-j}+\left( Z_{n-1}-c+a\right) w^{-2j}}{w^{-3j}+w^{-2j}-1}.
\end{eqnarray*}%
By considering the equality%
\begin{equation*}
\prod\limits_{k=0}^{n-1}\left( x-yw^{-k}+zw^{-2k}\right) =x^{n}\left(
1-\left( \frac{y-\sqrt{y^{2}-4xz}}{2x}\right) ^{n}-\left( \frac{y+\sqrt{%
y^{2}-4xz}}{2x}\right) ^{n}+\left( \frac{z}{x}\right) ^{n}\right) ,
\end{equation*}%
\ we have%
\begin{equation*}
\det (C(Z))=\frac{\left( Z_{n}-a\right) ^{n}\left(
1-K^{n}-L^{n}+K^{n}L^{n}\right) }{\left( -1\right) ^{n}\left(
Q_{-n}-Q_{n}\right) }.
\end{equation*}%
Hence proof is completed.
\end{proof}

By using in Theorem 9, we obtain determinants of circulant matrices with
Cordonnier, Perrin and Van Der Laan sequences as follows.

\begin{corollary}
$i)$ Let $a=b=c=1$ in Theorem 9. Then, the determinant \ of \ the matrix $%
C(P)=circ\left( P_{1},P_{2},\ldots ,P_{n}\right) $ is written by 
\begin{equation*}
\det (C(P))=\frac{\left( P_{n}-1\right) ^{n}\left(
1-K^{n}-L^{n}+K^{n}L^{n}\right) }{\left( -1\right) ^{n}\left(
Q_{-n}-Q_{n}\right) },
\end{equation*}%
where $K=\dfrac{1-P_{n+1}-\sqrt{\left( 1-P_{n+1}\right)
^{2}-4P_{n}P_{n-1}+4P_{n-1}}}{2\left( P_{n}-1\right) }$,
\end{corollary}

$L=\dfrac{1-P_{n+1}+\sqrt{\left( 1-P_{n+1}\right) ^{2}-4P_{n}P_{n-1}+4P_{n-1}%
}}{2\left( P_{n}-1\right) }$.

$ii)$ If we take $a=3,b=0,c=2$ in Theorem 9. Then, the determinant \ of \
the matrix $C(Q)=circ\left( Q_{1},Q_{2},\ldots ,Q_{n}\right) $ is written by 
\begin{equation*}
\det (C(Q))=\frac{\left( Q_{n}-3\right) ^{n}\left(
1-K^{n}-L^{n}+K^{n}L^{n}\right) }{\left( -1\right) ^{n}\left(
Q_{-n}-Q_{n}\right) },
\end{equation*}%
where $K=\dfrac{-Q_{n+1}-\sqrt{Q_{n+1}^{2}-4Q_{n}Q_{n-1}-4Q_{n}+12Q_{n-1}+12}%
}{2\left( Q_{n}-3\right) }$ and

$L=\dfrac{-Q_{n+1}+\sqrt{Q_{n+1}^{2}-4Q_{n}Q_{n-1}-4Q_{n}+12Q_{n-1}+12}}{%
2\left( Q_{n}-3\right) }$.

$iii)$Let $a=c=0,b=1$ in Theorem 9. Then, the determinant \ of \ the matrix $%
C(R)=circ\left( R_{1},R_{2},\ldots ,R_{n}\right) $ is written by 
\begin{equation*}
\det (C(R))=\frac{R_{n}^{n}\left( 1-K^{n}-L^{n}+K^{n}L^{n}\right) }{\left(
-1\right) ^{n}\left( Q_{-n}-Q_{n}\right) },
\end{equation*}%
where $K=\dfrac{1-R_{n+1}-\sqrt{\left( 1-R_{n+1}\right) ^{2}-4R_{n}R_{n-1}}}{%
2R_{n}}$,

$L=\dfrac{1-R_{n+1}+\sqrt{\left( 1-R_{n+1}\right) ^{2}-4R_{n}R_{n-1}}}{2R_{n}%
}$.

\end{document}